\documentclass[11pt]{amsart}
\usepackage{amssymb}
\usepackage{verbatim}
\usepackage{amsmath}
\usepackage{bm}
\usepackage{amssymb,latexsym}
\usepackage{enumerate}

\newcommand{\dbar}{\ensuremath{\overline\partial}}

\newcommand{\C}{\ensuremath{\mathbb{C}}}

\makeatletter
\newcommand{\sumprime}{\if@display\sideset{}{'}\sum%
            \else\sum'\fi}
\makeatother

\begin{document}

\numberwithin{equation}{section}

\newtheorem{theorem}{Theorem}[section]
\newtheorem{proposition}[theorem]{Proposition}
\newtheorem{conjecture}[theorem]{Conjecture}
\def\theconjecture{\unskip}
\newtheorem{corollary}[theorem]{Corollary}
\newtheorem{lemma}[theorem]{Lemma}
\newtheorem{observation}[theorem]{Observation}
\newtheorem{definition}{Definition}
\numberwithin{definition}{section} 
\newtheorem{remark}{Remark}
\def\theremark{\unskip}
\newtheorem{kl}{Key Lemma}
\def\thekl{\unskip}
\newtheorem{question}{Question}
\def\thequestion{\unskip}
\newtheorem{example}{Example}
\def\theexample{\unskip}
\newtheorem{problem}{Problem}

\thanks{Research supported by Knut and Alice Wallenberg Foundation, and the China Postdoctoral Science Foundation.}

\address{School of Mathematical Sciences, Fudan University, Shanghai, 200433, China}
\address{Current address: Department of Mathematical Sciences, Chalmers University of Technology and
University of Gothenburg. SE-412 96 Gothenburg, Sweden}
\email{wangxu1113@gmail.com}
\email{xuwa@chalmers.se}

\title[Higgs bundle and the K\"ahler cone]{A flat Higgs bundle structure on the complexified K\"ahler cone}
 \author{Xu Wang}
\date{\today}

\begin{abstract} We shall construct a natural Higgs bundle structure on the complexified K\"ahler cone of a compact K\"ahler manifold, which can be seen as an analogy of the classical Higgs bundle structure associated to a variation of  Hodge structure. In the proof of the flat-ness of our Higgs bundle, we find a commutator identity that can be used to decode the variational properties of the polarized Hodge-Lefschetz module structure on the fibres of our Higgs bundle. Thus we can use a generalized version of Lu's Hodge metric to study the curvature property of the complexified K\"ahler cone. In particular, it implies that the above Hodge metric defines a K\"ahler metric on the complexified K\"ahler cone with negative holomorphic sectional curvature, which can be seen as a new result on Wilson's conjecture.  

\bigskip


\smallskip

\end{abstract}
\maketitle

\tableofcontents

\section{Introduction}

\subsection{Motivation}

Our motivation is to study the geometric properties of a natural Higgs bundle on a complexified K\"ahler cone. Let $X$ be an $n$-dimensional compact K\"ahler manifold. The K\"ahler cone, say $\mathcal K$, of $X$ is the space of K\"ahler classes in 
$$
H^{1,1}(X, \mathbb R):=H^2(X,\mathbb R) \cap H^{1,1}(X, \mathbb C).
$$
Thus $\mathcal K$ is an open convex cone in $H^{1,1}(X, \mathbb R)$. $\mathcal K$ has a natural complexification
\begin{equation}
\mathcal K_{\C}:=\mathcal K+\sqrt{-1} H^{1,1}(X, \mathbb R).
\end{equation}
Let us consider the trivial bundle, say $H$, over $\mathcal K_\C$ with fibre 
$$
\oplus_{p,q=0}^n H^{p,q}(X,\C).
$$
Thus $H$ is a direct sum of the following trivial vector bundles
\begin{equation}
H^{p,q}:= \mathcal K_\C \times H^{p,q}(X,\C),
\end{equation}
over our complexified K\"ahler cone. There is a natural non-trivial Hermitian metric, say $h$, on $H$ defined by using the Lefschetz decomposition as follows:

\medskip

\textbf{Hermitian metric on $H$}: Let $\omega+\sqrt{-1} \alpha$ be a point in $\mathcal K_\C$, i.e. $\omega$ has a K\"ahler form, say $\bm{\omega}$, as a representative. Then we shall define the Hermitian metric, say $h$ on $H$, at the point $\omega+\sqrt{-1} \alpha$, as the Hermitian metric on $\oplus  H^{p,q}(X,\C)$ defined by $\bm\omega$. In order to see that $h$ is well defined, i.e. $h$ does not depend on the choice of representatives $\bm{\omega}$, we have to use another definition of $h$. By the Lefschetz decomposition, every $u\in H^{p,q}(X,\C)$ has a unique decomposition as follows
\begin{equation}
u=\sum \omega_r u^r, \ \omega_r:=\frac{\omega^r}{r!},  
\end{equation}
where each $u^r$ is a primitive class and $\omega_r u^r$ is the class with representative ${\bm \omega}_r \wedge \bm{u}^r$, where $\bm{u}^r$ is an arbitrary $d$-closed representative of $u^r$. Then we can define the Hodge star operator on $H$:
\begin{equation}
*u:=i^{k^2} \sum (-1)^{p-r} \omega_{n-k+r}u^r,\  k:=p+q,\  i:=\sqrt{-1}. 
\end{equation}
Thus we have
\begin{equation}
h(u,u)=u\overline{*u}(X):=i^{k^2} \int_X  \sum (-1)^{q-r} {\bm \omega}_r \wedge {\bm \omega}_{n-k+r} \wedge \bm{u}^r \wedge \overline{\bm{u}^r}.
\end{equation}
By the classical Hodge-Riemann bilinear relations, we know that $h$ is a well-defined Hermitian metric on $H$. A natural question is 

\medskip

\emph{Is there a flat Higgs bundle structure on $(H,h)$ as an analogy of the natural Higgs bundle structure associated to a variation of Hodge structure ?}

\medskip

The reason why we ask such a question is if there exists such a flat Higgs bundle on $(H,h)$ then our main result in \cite{Wang-Higgs} can be used to study the following conjecture of Wilson in \cite{Wilson04} (see also \cite{TW11} and \cite{M12}), i.e.

\medskip

\emph{Whether the sectional curvature of the level set of the real K\"ahler cone (with the Weil-Petersson type metric) is between $-\frac{n(n-1)}{2}$ and zero, at least for Calabi-Yau manifolds?} 

\medskip

In section 3 of  \cite{TW11}, Trenner and Wilson gave a counterexample to the above conjecture, which suggests to find other natural metrics on the K\"ahler cone instead of the Weil-Petersson type metric  to study the geometry of the K\"ahler cone. In variation of Hodge structure case, we know that Lu's metric (see \cite{Lu99} and \cite{Lu01}) has negative sectional curvature. In \cite{Wang-Higgs}, we found that one may also define Lu's Hodge metric for a general Higgs bundle, and if the Higgs bundle is flat then the associated Hodge metric will have semi-negative holomorphic bisectional curvature. 

\medskip

In this paper, we shall study the above conjecture by constructing a natural flat Higgs bundle structure on $(H,h)$ (see \cite{Deligne79} and \cite{Gross94} for early related results). Then \cite{Wang-Higgs} implies that there is a K\"ahler metric on the complexified K\"ahler cone with negative holomorphic sectional curvature. Another (maybe more interesting) result in this paper is:

\medskip

\emph{In the proof of the flat-ness of our Higgs bundle, we find a fundamental commutator identity for the variational properties of the polarized Hodge-Lefschetz module structure on the fibres of our Higgs bundle. This commutator identity is widely true for general polarized Hodge-Lefschetz module (see section 1.4 for its relation with the usual K\"ahler identity).}

\subsection{Higgs field structure on the K\"ahler cone}

Notice that the holomorphic tangent bundle of $\mathcal K_\C$ is a trivial bundle with fibre $H^{1,1}(X,\C)$. Then
\begin{equation}
\alpha: u\mapsto \alpha u, \ \alpha\in H^{1,1}(X,\C), \ u\in H^{p,q}(X,\C),
\end{equation}
defines a natural action of $H^{1,1}(X,\C)$ on the fibres of $H$. Thus we have a constant (thus holomorphic) bundle map, say
\begin{equation}
\theta: T\mathcal K_\C \to {\rm End}(H), 
\end{equation}
such that
\begin{equation}
\theta(\alpha)(u):=\alpha u.
\end{equation}
Let us look at $\theta$ as an ${\rm End}(H)$-valued holomorphic one form on $\mathcal K_\C$ and still denote it by $\theta$. Since $\alpha\beta u=\beta\alpha u$, then we know that $\theta^2=0$. By Hitchin-Simpson's definition (see \cite{Hitchin87} and \cite{Simpson92}), we know that $\theta$ is a Higgs field and $(H,\theta)$ is a Higgs bundle (in the classical variation of Hodge structure setting, the Higgs field is defined by the Kodaira-Spencer action, see the appendix in \cite{Wang-Higgs}). Our first main result is a proof of the flat-ness of $(H,h)$. 

\begin{theorem}\label{th:cone-higgs} Let $D^h=\dbar+\partial^h$ be the Chern connection on $(H,h)$. Then the Higgs connection $D^H:=D^h+\theta+\theta^*$ is flat, i.e. $\Theta^H:=(D^H)^2=0$. In particular, it implies that the curvature, $\Theta^h:=(D^h)^2$, of the Chern connection on $H$ satisfies $\Theta^h+\theta\theta^*+\theta^*\theta=0$.
\end{theorem}

\subsection{An application on Wilson's conjecture}

By our main result in \cite{Wang-Higgs}, we know that the above theorem implies the following theorem, which can be seen as a new result on Wilson's conjecture, see the abstract of \cite{Wilson04} and Discussion 4.4 in \cite{Wilson04}):

\begin{theorem}\label{th:cone-curvature} Lu's Hodge metric (pull back to $T\mathcal K_C$ of the metric on ${\rm End}(H)$ by $\theta$) defines a K\"ahler metric on $\mathcal K_C$ with semi-negative holomorphic bisectional curvature.  Moreover, its holomorphic sectional curvature is bounded above by $-(n^2{\rm Rank}(H))^{-1}$.
\end{theorem}

\textbf{Remark 1}: There is also a Weil-Petersson type metric (see \cite{Huybrechts01} and \cite{Wilson04}) on $\mathcal K_\C$ defined by the Hermitian metric on $H^{1,1}(X,\C)$. But it is known that in general the Weil-Petersson metric on the moduli space of a Calabi-Yau manifold does not have negative curvature property (see \cite{CXGP} for a counterexample, see also \cite{TW11} for a counterexample for the K\"ahler cone). Our Hodge metric in Theorem \ref{th:cone-curvature} is a generalization of Lu's Hodge metric in \cite{Lu99} and \cite{Lu01}. For early results about the curvature property of
Lu's Hodge metric, see \cite{GS69}, \cite{Deligne69} and \cite{G84}; see also \cite{LS04}, \cite{FL05}, \cite{PP14} and references therein for recent developments.

\medskip

\textbf{Remark 2}: One may also get a better estimate of the holomorphic sectional curvature by considering the flat Higgs subbundles
\begin{equation}
H^k:=\oplus_{p-q=k} H^{p,q}, \  -n\leq k\leq n,
\end{equation}
of $H$. In particular, since $H^{0}$ is admissible (see \cite{Wang-Higgs}), we know that the holomorphic sectional curvature of the Hodge metric associated to $H^0$ is bounded above by $\frac{-4}{n^2 {\rm Rank}(H_0)}$ (the reason why we have a number $4$ here is: $\dim_\C H^{p,p}(X,\C)=\dim_\C H^{n-p,n-p}(X, \C)$, thus we can have a better estimate in the proof of our main result in \cite{Wang-Higgs}).

\subsection{A commutator identity behind the proof}

Since the metric $h$ on our Higgs bundle $H$ is fully determined by the polarized Hodge-Lefschetz structure on $(\oplus H^{p,q}(X,\C), \omega, h)$. We know that the curvature property of the Chern connection on $H$ will be a variational property of the polarized Hodge-Lefschetz structure on the fibres of $H$. In our computations, we find that, one may use a commutator identity like, $[*, \partial/\partial t^j]$, to decode the curvature properties of our Higgs bundle.

\medskip

\textbf{Motivation}: Let us fix a base, say $\{e_j\}_{1\leq j \leq N}$ of $H^{1,1}(X,\mathbb R)$. Then our complexified K\"ahler cone, $\mathcal K_\C$, can be seen as a convex tube domain, say $K_\C=K+i\mathbb R^N$, in $\C^N$. Let $$z^j:= t^j+i s^j, \ 1\leq j\leq N, $$ be the classical coordinate system on $\C^N$. Then we know that $z=t+is\in K_\C$ if and only if 
\begin{equation}\label{eq:110}
\omega(t):=\sum t^j e_j,
\end{equation}
has a K\"ahler form as a representative. Now our Hermitian metric on $H$ is fully determined by the polarized Hodge-Lefschetz structure associated to the Lefschetz action of 
\begin{equation}
\omega: z=t+is \mapsto \omega(t), 
\end{equation}
on $H$. Let us write the Chern connection on $(H, h)$ as
\begin{equation}
D^h=\dbar+\partial^h=\sum d\bar z^j \wedge \frac{\partial}{\partial \bar z^j}+\sum dz^j \wedge \partial_{z^j}^h.
\end{equation}
Then each $\partial^h_{z^j}$ is determined by
\begin{equation}\label{eq:chern}
\frac{\partial}{\partial z^j} (u,v)=(\partial_{z^j}^h u, v)+(u, \partial v/\partial \bar z^j).
\end{equation}
If we look at $\partial_{z^j}^h$ and $\partial / \partial \bar z^j$ as differential operators on the space of smooth sections of $H$ then \eqref{eq:chern} implies that $-\partial_{z^j}^h$ is just the adjoint of $\partial / \partial \bar z^j$ with respect to $h$ and the classical Euclidean metric on $K_\C$, i.e.
\begin{equation}\label{eq:basic}
(\partial / \partial \bar z^j)^*=-\partial^h_{z^j}.
\end{equation} 
Recall the curvature of the Chern connection on $H$ is
\begin{equation}
\Theta^h:=(D^h)^2= \sum [\partial_{z^j}^h,  \frac{\partial}{\partial \bar z^k}]  dt^j\wedge d\bar t^k.
\end{equation}
Thus \eqref{eq:basic} implies that it is necessary to find a good formula of $(\partial / \partial \bar z^j)^*$ to decode $\Theta^h$. Let us write our Higgs field $\theta$ as $\theta:= \sum dz^j \wedge \theta_{z_j}$. Then we have
\begin{equation}
\theta_{z^j}=[\partial/\partial z^j, \omega]=\partial \omega/\partial z^j =\frac12 e_j,
\end{equation}
where the last equality follows from \eqref{eq:110}. Let $\Lambda$ be the adjoint of $\omega: u\mapsto \omega u$ and $\tau:=**$. Our fundamental identity reads as follows (here each $\partial/\partial z^j$ is a differential operator, $(\partial / \partial \bar z^j)^*$ denotes the adjoint operator, nothing to do with the star operator):

\begin{theorem}\label{th:main} $-(\partial / \partial \bar z^j)^*=\partial^h_{z^j}=\tau * (\partial/\partial z^j)*=[\Lambda, \theta_{z^j}]+\partial/\partial z^j$.
\end{theorem}

\textbf{Remark}: Since $\tau^2=1$, we know that the above identity is equivalent to 
\begin{equation}\label{eq:fundamental-1}
[\partial/\partial z^j, *]=*[\Lambda, \theta_{z^j}].
\end{equation}
We can also define  $*$, $\omega$, $\Lambda$ and $Y:=[\omega,\Lambda]$ for a general polarized Hodge-Lefschetz module (see Definition 2.3 in \cite{Cattani08}, see also \cite{LL97}). Since $\{\omega, \Lambda, Y\}$ is an 
$sl_2$-triple, the commutators between them are clear. In general,  assume we have a smooth family of polarized Hodge-Lefschetz module structures whoes associated $sl_2$-triples are $\{\omega(t), \Lambda(t), Y(t)\}_{t\in \mathbb R}$ (with $*(t)$ as the corresponding star operators), put
\begin{equation}
\{\omega, \Lambda, Y, *\}: t\mapsto \{\omega(t), \Lambda(t), Y(t), *(t)\}
\end{equation}

\begin{definition} We call the commutators between $\partial/\partial t$ and $\{\omega, \Lambda, Y, *\}$ the fundamental commutators.
\end{definition}

\textbf{Remark}: Since
\begin{equation}
[\partial/\partial t, Y]=0, \ [\partial/\partial t, \omega]=\partial\omega/\partial t.
\end{equation}
we only have two non-trivial fundamental commutators. In the third section, we shall prove that \eqref{eq:fundamental-1} can also be used to prove the following identity:

\begin{proposition}\label{pr:main-1} $\theta_{z^k}^*=[\Lambda, \frac{\partial}{\partial \bar z^k}]=-\frac12 [\Lambda, [\Lambda, \overline{\theta_{z^k}}]]$.
\end{proposition}

Thus the identity in Theorem \ref{th:main} is essentially the only 
fundamental commutator. Let us compare it with the classical K\"ahler identity (see \cite{Demailly} and \cite{Wells} for the proof, see also \cite{Weil})
\begin{equation}\label{eq:kahler}
\dbar^*= i[\partial, \Lambda],
\end{equation}
on a K\"ahler manifold. 
The following trivial identity
\begin{equation}
\dbar=\sum d\bar z^j \wedge \frac{\partial}{\partial \bar z^j}.
\end{equation}
builds the bridge between \eqref{eq:kahler} and our K\"ahler identity. Because then, we can write
\begin{equation}
\dbar^*=\sum (\frac{\partial}{\partial \bar z^j})^* ( d\bar z^j \wedge)^*.
\end{equation}
Thus the adjoint of $\dbar$ is essentially a composition of two operators:  $(\frac{\partial}{\partial \bar z^j})^*$ and $ (d\bar z^j \wedge)^*$. It is well known that (see formula 1 in \cite{OT87} and its applications)
\begin{equation}\label{eq:k-id-1}
( d\bar z^j \wedge)^*= i [\Lambda, (dz^j\wedge)].
\end{equation}
Thus one may say that our identity in Theorem \ref{th:main} is more primitive than the usual K\"ahler identity. A first application of Theorem \ref{th:main} is a quick proof of Theorem \ref{th:cone-higgs}.

\subsection{Proof of Theorem \ref{th:cone-higgs}} By Theorem \ref{th:main} and Proposition \ref{pr:main-1}, we have
\begin{equation}
 [\partial_{z^j}^h,  \frac{\partial}{\partial \bar z^k}]=[[\Lambda, \theta_{z^j}], \frac{\partial}{\partial \bar z^k}]=[[\Lambda,\frac{\partial}{\partial \bar z^k}], \theta_{z^j}]=[\theta_{z^k}^*, \theta_{z^j}],
\end{equation}
which is equivalent to $\Theta^h+\theta\theta^*+\theta^*\theta=0$. By a similar argument, we also have $\partial^h\theta +\theta\partial^h=0$ and $\dbar\theta^*+\theta^*\dbar=0$. Thus we know that $(D^h)^2=(\dbar+\partial^h+\theta+\theta^*)^2=0$. The proof of Theorem \ref{th:cone-higgs} is complete.

\medskip

\textbf{Remark}: Just like the variation of Hodge structure case, the following corollary a direct consequence of Theorem \ref{th:cone-higgs}.
 
\begin{corollary} The curvature $\Theta^h$ of the connection on $H^{0,0}$ is $-\theta^*\theta$. In particular, it implies that $H^{0,0}$ is a positive line bundle.
\end{corollary}

In section 2.1, we shall show that the curvature of $H^{0,0}$ can be seen as an effective version of the Brunn-Minkowski type inequalities on a compact K\"ahler manifold. There is also another deep and beautiful complex version of the Brunn-Minkowski theory of Berndtsson, see \cite{Bern06} and \cite{Bern09}. We shall study their relations in \cite{Wang-Kahler}, in particular, we shall show how to use Theorem \ref{th:main}  to explain our main results in \cite{Wang16} and \cite{BPW16}. 

\subsection{Generalizations and other related results} The proof of Theorem \ref{th:main} will be given in the third section. In a future publication \cite{Wang-Kahler}, we shall study generalizations of Theorem \ref{th:main} for a general smooth family of polarized Hodge-Lefschetz module structures. There is also a mixed version of the Hodge-Riemann bilinear relation. Its origin lies in the convex case, i.e. the so called Alexandrov-Fenchel inequality (see the next section for the background). There is also an Alexandrov-Fenchel inequality for the mixed discriminant, first proved by Alexandrov, and then it was generalized by Timorin in \cite{Timorin98} to a mixed Hodge-Riemann bilinear relation associated to the exterior algebra of a fixed point in a Hermitian manifold. Later, Dinh and Nguy\^en \cite{DN06} proved that Timorin's result is also true globally, more precisely, they gave a mixed Hodge-Riemann bilinear relation for every compact K\"ahler manifold. Dinh-Nguy\^en's result was further generalized to general polarized Hodge-Lefschetz modules by Cattani in \cite{Cattani08}. The relation between the convex case and the compact K\"ahler case was first studied by Khovanskii \cite{Khovanskii88} and Gromov \cite{Gromov90}, see \cite{McMullen93},  \cite{KK12} and \cite{KK12-1} and references therein for other interesting developments.

\subsection{Acknowledgements} 

I would like to thank Professor Bo Berndtsson for many inspiring discussions on the Alexandrov-Fenchel inequality, which is the starting point of this paper. I would also like to thank Professor Takeo Ohsawa, who kindly gave me a background of our identity in Theorem \ref{th:main}: in 1990 (ICM Kyoto), Professor Jean-Michel Kantor has already talked with him about our identity. I would also like to thank Professor Mihai P\u aun for sending me the reference \cite{M12}. Thanks are also given to Professor Bo-Yong Chen and Professor Qing-Chun Ji for their constant support and encouragement.

\section{Background and examples}

\subsection{Brunn-Minkowski inequality and curvature of $H^{0,0}$} We believe that this subsection is essentially well known, but we shall still write it down as the background and starting point of this paper. We shall follow Berndtsson's notes on convex and complex geometry (available in his homepage) in this section. Let us first recall the classical Brunn-Minkowski inequality. 

\medskip

\textbf{Classical Brunn-Minkowski theory}: Let $A_0$, $A_1$ be two convex bodies in $\mathbb R^n$, i.e. compact convex sets with non-empty interior. Their Minkowski sum is defined as follows
\begin{equation*}
A_0+A_1:=\{a_0+a_1: a_0\in A_0, \ a_1\in A_1\}.
\end{equation*}
The Brunn-Minkowski theorem reads as follows:

\begin{theorem}[Brunn-Minkowski inequality] $|A_0+A_1|^{1/n} \geq |A_0|^{1/n} +|A_1|^{1/n}$, where the absolute value of a convex body means its volume (Lebesgue measure).  
\end{theorem} 

Put $A_t:=tA_1+(1-t)A_0$, then the above theorem is equivalent to that
\begin{equation}
t\mapsto -|A_t|^{1/n},
\end{equation}
is convex on $(0,1)$. Since $|cA_t|=c^n |A_t|$, we also know that the Brunn-Minkowski theorem is equivalent to that
\begin{equation}
t\mapsto -\log |A_t|,
\end{equation}
is convex. On the other hand, if $A$ is a convex body in $\mathbb R^n$ and $\phi$ is a strictly convex function of class $C^2$  on $\mathbb R^n$ such that
the image of
\begin{equation}
\partial \phi: x\mapsto y=\partial \phi(x):=(\partial\phi/\partial x^1, \cdots, \partial\phi/\partial x^n),
\end{equation}
is just the interior of $A$ (for this, it suffices to take $\phi$ as the Legendre transform of a smooth strictly convex function in $A$ that tends to infinity at the boundary of $A$), then we have
\begin{equation}
|A|=\int_{A} dy=\int_{\mathbb R^n} MA(\phi) dx, \ dx:=dx^1\wedge \cdots \wedge dx^n, \  dy:=dy^1\wedge \cdots \wedge dy^n.
\end{equation} 
where $MA(\phi)$ denotes the determinant of the Hessian of $\phi$. In general, we have
\begin{equation}
p(t):=|t_1 A_1+\cdots +t_n A_n|=\int_{\mathbb R^n} MA(t_1\phi_1+\cdots +t_n\phi_n) dx.
\end{equation}
We call the coefficent of $t_1\cdots t_n$ in the polynomial $p(t)$ the mixed volume, say $V(A_1, \cdots, A_n)$, of $A_1, \cdots, A_n$. Then we have the following Alexandrov-Fenchel theorem as a generalization of the Brunn-Minkowski theorem: 

\begin{theorem}[Alexandrov-Fenchel inequality] $$V(A_1, \cdots, A_n)^2\geq V(A_1,A_1, A_3,\cdots, A_n) V(A_2,A_2, A_3,\cdots, A_n).$$  
\end{theorem} 

Put $A_t:= tA_1+(1-t)A_2$, then the above inequality is equivalent to the convexity of
\begin{equation}
t\mapsto -V(A_t, A_t, A_3,\cdots, A_n)^{1/2},
\end{equation}
which is also equivalent to the convexity of
\begin{equation}
t\mapsto -\log V(A_t, A_t, A_3,\cdots, A_n).
\end{equation}
There is also an analogy of the above theory for compact K\"ahler manifold (first studied by Khovanskii and Teissier). 

\medskip

\textbf{Brunn-Minkowski theory for compact K\"ahler manifold}: Let $X$ be a compact K\"ahler  manifold. Let $\omega_1, \cdots, \omega_n$ be K\"aher classes on $X$. Put
\begin{equation}
V(\omega_1, \cdots, \omega_n):=\int_X \bm \omega_1\wedge\cdots \wedge \bm\omega_n.
\end{equation}
Then we have the following Brunn-Minkowski theorem for compact K\"ahler manifold (compare with Minkowski's second inequality in the convex case):

\begin{theorem} $$V(\omega_1, \omega_2, \cdots, \omega_2)^2\geq V(\omega_1,\omega_1, \omega_2,\cdots, \omega_2) V(\omega_2, \cdots, \omega_2).$$  
\end{theorem} 

Put $\omega(t):=(1-t) \omega _1+t\omega_2$. Then we know that the above inequality is equivalent to the convexity of 
\begin{equation}
t\mapsto -\log V(\omega(t), \cdots, \omega(t)).
\end{equation}
In general we have the following Khovanskii-Teissier theorem:

\begin{theorem}[Khovanskii-Teissier inequality] $$V(\omega_1, \cdots, \omega_n)^2\geq V(\omega_1,\omega_1, \omega_3,\cdots, \omega_n) V(\omega_2,\omega_2, \omega_3,\cdots, \omega_n).$$  
\end{theorem} 

By the same reason, we know that the Khovanskii-Teissier inequality is  equivalent to the convexity of
\begin{equation}
t\mapsto -\log V(\omega(t),\omega(t), \omega_3,\cdots. \omega_n).
\end{equation}
In the next section, we shall show how to look at the above convexity properties as positivity properties of the curvature of $H^{0,0}$. In this way, we can move one step further, i.e. we can prove that the second order derivatives of the above functions like $-\log V(\omega(t), \cdots, \omega(t))$ are fully determined by the norm of our Higgs field.

\medskip

\textbf{Curvature of $H^{0,0}$}: Recall that our bundle $H^{0,0}$ is a trivial line bundle over the complexified K\"ahler cone $\mathcal K_\C$. Let us fix a base, say $\{e_j\}_{1\leq j \leq N}$ of $H^{1,1}(X,\mathbb R)$. Then $\mathcal K_\C$, can be seen as a convex tube domain (Siegel domain of the first kind), say $K_\C$, in $\C^N$. And we know that $z=t+is\in K_\C$ if and only if 
\begin{equation}
\omega(t):=\sum t^j e_j,
\end{equation}
has a K\"ahler form as a representative. We know that each $\omega(t)$ defines a Hermitian metric on $H^{p,q}(X,\C)$. In particular, 
\begin{equation}
\omega: z=t+is \mapsto \omega(t),  \ z\in K_\C,
\end{equation}
defines a Hermitian metric, say $h$, on $H^{0,0}$. Let us write the  Chern connection on $(H, h)$ as
\begin{equation}
D^h=\dbar+\partial^h=\sum d\bar z^j \wedge \frac{\partial}{\partial \bar z^j}, \ \partial^h=\sum dz^j \wedge \partial_{z^j}^h.
\end{equation}
Thus
\begin{equation}
\Theta^h:=(D^h)^2=\sum [\partial_{z^j}^h, \frac{\partial}{\partial \bar z^k}] dz^j \wedge d\bar z^k.
\end{equation}
Moreover, since $\omega$ does not depend on $s$, we have
\begin{equation}
 \frac{\partial}{\partial \bar z^j}=\frac12( \frac{\partial}{\partial  t^j}+ i \frac{\partial}{\partial s^j}), \ \partial_{z^j}^h=\frac12(\partial_{t^j}^h- i \frac{\partial}{\partial s^j}),
\end{equation}
where each $\partial_{t^j}^h$ is determined by
\begin{equation}
\frac{\partial}{\partial t^j} (u,v)=(\partial_{t^j}^h u, v)+(u, \partial v/\partial t^j).
\end{equation}
Thus we have
\begin{equation}
\partial_{t^j}^h=* \frac{\partial}{\partial t^j} *,
\end{equation}
and
\begin{equation}
\Theta^h:=\frac14 \sum \Theta_{jk}^h dz^j \wedge d\bar z^k, \ \Theta_{jk}^h:=[\partial_{t^j}^h, \frac{\partial}{\partial  t^k}]. 
\end{equation}
Let us denote by $1$ the canonical frame of $H^{0,0}$. We shall compute
\begin{equation}
\sum (\Theta_{jk}^h 1, 1)\zeta^j \zeta^k, \ \ \zeta \in \mathbb R^N. 
\end{equation}
Notice that
\begin{equation}
\Theta_{\zeta\zeta}:=\sum \Theta_{jk}^h \zeta^j \zeta^k=[\partial^h_\zeta, \partial_\zeta], \ \  \partial_\zeta:=\sum \zeta^j\frac{\partial}{\partial  t^j}, \ \partial^h_\zeta=*\partial_\zeta*.
\end{equation}
Put
\begin{equation}\label{eq:theta-def}
\theta_\zeta:=[\partial_\zeta, \omega]=\sum \zeta^je_j.
\end{equation}
We shall prove the following theorem:

\begin{theorem} $\Theta_{\zeta\zeta}=\theta_\zeta^* \theta_\zeta$.
\end{theorem}

\begin{proof}  Since the canonical frame $1$ of $H^{0,0}$ is holomorphic, we have
\begin{equation}
(\Theta_{\zeta\zeta} 1, 1)=||\partial^h_\zeta 1||^2-(1,1)_{\zeta\zeta}.
\end{equation}
Since $\partial^h_\zeta$=$*\partial_\zeta*$, we have
\begin{equation}
||\partial^h_\zeta 1||^2=||*(\theta_\zeta \omega_{n-1})||^2.
\end{equation}
Let
\begin{equation}
\theta_\zeta=\omega a_\zeta+b_\zeta,
\end{equation}
be the primitive decomposition of $\theta_\zeta$. Then we have
\begin{equation}
||\partial^h_\zeta 1||^2=n^2a_\zeta^2|X|, \ |X|:=\int_X \omega_n.
\end{equation}
On the other hand, we have
\begin{equation}
(1,1)_{\zeta\zeta}=|X|_{\zeta\zeta}=\int_X \theta_\zeta^2 \omega_{n-2}=\int_X (\omega a_\zeta+b_\zeta)^2\omega_{n-2} =n(n-1)a_\zeta^2 |X|-||b_\zeta||^2.
\end{equation}
Thus
\begin{equation}
(\Theta_{\zeta\zeta} 1, 1)=na_\zeta^2 |X|+||b_\zeta||^2=||\omega a_\zeta||^2+||b_\zeta||^2=||\theta_\zeta||^2.
\end{equation}
The proof is complete.
\end{proof}

\textbf{Remark}: The above theorem implies that $H^{0,0}$ is a positive line bundle and 
\begin{equation}
(-\log |X|)_{\zeta\zeta} =(-\log ||1||^2)_{\zeta\zeta} =\theta_\zeta^* \theta_\zeta=||\theta_\zeta||^2/|X|,
\end{equation}
which can be seen as an effective version of the Brunn-Minkowski inequality on compact K\"ahler manifold. Since $H^{n,n}$ is the dual bundle of $H^{0,0}$, thus at least in one dimensional case, we can guess and check that our total bundle $H$ is a flat Higgs bundle.

\subsection{Elementary proof for the two dimensional case}

In this section, we shall compute the curvature of $H^{1,1}$ in case $n=2$. Notice that the fibre of $H^{1,1}$ is spanned by $\{e_j\}$, thus by  \eqref{eq:theta-def}, we know that it is enough to compute $(\Theta_{\zeta\zeta} \theta_\eta, \theta_\eta)$, where we identify $\theta_\eta$ with $\theta_\eta 1$. Our aim is to prove the following theorem:

\begin{theorem}\label{th:n2-ele} $(\Theta_{\zeta\zeta} \theta_\eta, \theta_\eta)=R:=|| \theta_\zeta \theta\eta||^2-||\theta_\zeta^*\theta_\eta||^2$.
\end{theorem}

Still let
\begin{equation}
\theta_\zeta=\omega a_\zeta+b_\zeta,
\end{equation}
be the primitive decomposition of $\theta_\zeta$. Let us define $b_{\zeta\eta}$ such that
\begin{equation}
b_\zeta b_\eta=b_{\zeta\eta}\omega^2.
\end{equation}
Thus we have
\begin{equation}
|| \theta_\zeta \theta\eta||^2=4(b_{\zeta\eta}+a_\eta a_\zeta)^2|X|, \  ||\theta_\zeta^*\theta_\eta||^2=4(b_{\zeta\eta}-a_\eta a_\zeta)^2|X|,
\end{equation}
which implies that
\begin{equation}
R=16 a_\eta a_\zeta b_{\zeta\eta}|X|.
\end{equation}
Since each $\theta_\eta$ is a constant section, we have
\begin{equation}
(\Theta_{\zeta\zeta} \theta_\eta, \theta_\eta)=||\partial^h_\zeta \theta_\eta||^2-(||\theta_\eta||^2)_{\zeta\zeta}. 
\end{equation}
Thus it is enough to prove that
\begin{equation}\label{eq:16}
||\partial^h_\zeta \theta_\eta||^2-(||\theta_\eta||^2)_{\zeta\zeta}=16 a_\eta a_\zeta b_{\zeta\eta}|X|.
\end{equation}
The following inequalities will be crucial in our computations:

\begin{proposition}\label{pr:n=2} Let us write $a_{\eta,\zeta}=\partial_\zeta a_\eta$ and $b_{\zeta\eta, \lambda}=\partial_\lambda b_{\zeta\eta}$. Then we have
\begin{equation}
b_{\zeta\eta, \lambda}=-a_\zeta b_{\eta\lambda} -a_\eta b_{\zeta\lambda} -2 a_\lambda b_{\zeta\eta}, \ a_{\eta,\zeta}=b_{\zeta\eta}-a_\eta a_\zeta.
\end{equation}
\end{proposition}

We leave it as an exercise for the reader to use the above proposition to prove the following two lemmas:

\begin{lemma} $||\partial^h_\zeta \theta_\eta||^2=8(b_{\zeta\eta}^2-a_\eta^2b_{\zeta\zeta})|X|$.
\end{lemma}

\begin{lemma} $(||\theta_\eta||^2)_{\zeta\zeta}=8(b_{\zeta\eta}^2-a_\eta^2b_{\zeta\zeta}-2a_\eta a_\zeta b_{\zeta\eta})|X|$.
\end{lemma}

Now it is clear that Theorem \ref{th:n2-ele} follows froms the above two lemmas. Thus it suffices to prove Proposition \ref{pr:n=2}.

\begin{proof}[Proof of Proposition \ref{pr:n=2}] Notice that $b_\zeta b_\eta=b_{\zeta\eta}\omega^2$ implies that
\begin{equation}
b_{\zeta,\lambda} b_\eta+
b_\zeta b_{\eta, \lambda}=
b_{\zeta\eta, \lambda}\omega^2
+2b_{\zeta\eta} \theta_\lambda\omega=b_{\zeta\eta, \lambda}\omega^2
+2b_{\zeta\eta} a_\lambda\omega^2.
\end{equation}
On the other hand, $\theta_{\zeta, \lambda}\equiv 0$ implies that
\begin{equation}
b_{\zeta,\lambda} =-(a_{\zeta,\lambda}\omega + a_\zeta \theta_\lambda).
\end{equation}
Thus
\begin{equation}
b_{\zeta,\lambda} b_\eta=-a_\zeta \theta_\lambda b_\eta=-a_\zeta b_{\lambda\eta}\omega^2,
\end{equation}
thus our first identity follows. Now let us prove the second one. Notice that $(\theta_\eta^2)_\zeta\equiv 0$ implies that
\begin{equation}
((a_\eta^2+b_{\eta\eta})\omega^2)_\zeta \equiv 0.
\end{equation}
Thus we have
\begin{equation}
(a_\eta^2+b_{\eta\eta})_\zeta+ 2(a_\eta^2+b_{\eta\eta})a_\zeta \equiv 0.
\end{equation}
By our first identity, we know that 
\begin{equation}
b_{\eta\eta, \zeta}= -2 a_\eta b_{\zeta\eta} -2 b_{\eta\eta} a_\zeta.
\end{equation}
The above two identities implies that
\begin{equation}
2a_\eta(a_{\eta,\zeta}-b_{\zeta\eta}-a_\eta a_\zeta) \equiv 0.
\end{equation}
Thus our second identity follows.
\end{proof}

\textbf{Remark}: As we can see, even in the two dimensional case, the proof of Theorem \ref{th:n2-ele} is not easy (in fact, we spent quite a lot of time to find the identities in Proposition \ref{pr:n=2}). We have also tried to find a higher dimensional version of Proposition \ref{pr:n=2}, but we failed. That is the main reason why we want to find an abstract proof of Theorem \ref{th:cone-higgs}. But Proposition \ref{pr:n=2} plays a crucial role in our finding of the identity in Theorem \ref{th:main}.

\section{Proof of Theorem \ref{th:main} and Proposition \ref{pr:main-1}}

We shall use the following two basic facts (see \cite{Demailly}, \cite{Wells} and \cite{Weil}) in our proof:

\medskip

\textbf{Fact 1}: If $u$ is a primitive $k$-class then
\begin{equation}
\Lambda(\omega_r u)=(n-k-r+1) \omega_{r-1} u, \ \ r \leq n-k+1.
\end{equation}

\medskip

\textbf{Fact 2}: If $u$ is a primitive $(p,q)$-class then
\begin{equation}
*(\omega_r u):= i^{k^2} (-1)^{p} \omega_{n-r-k} u, \ k:=p+q.
\end{equation}

\begin{proof}[Proof of Theorem \ref{th:main}]  Recall that $(u,v)=u\overline{*v}(X)$, thus \eqref{eq:chern} implies that $$\partial^h_{z^j}=\tau*(\partial/\partial z^j)*.$$ Thus it suffices to prove the last equality. By the Lefschetz decompostion, every smooth section of $H$ can be wriiten as
\begin{equation}
\sum \omega_r u^r,
\end{equation}
where each $u^r$ is primitive. Thus it is enough to prove Theorem \ref{th:main} for forms like $\omega_r u$, where $u$ is a primitive $(p,q)$-class. By \textbf{Fact 2}, we have
\begin{equation}
*(\omega_r u):= i^{k^2} (-1)^{p} \omega_{n-r-k} u, \ k:=p+q.
\end{equation}
Thus, applying the differential operator $\partial/\partial z^j$, we get 
\begin{equation}
(\partial/\partial z^j) *(\omega_r u)= i^{k^2} (-1)^{p}(\omega_{n-r-k} u_{j}+  \omega_{n-r-k-1} \theta_{z^j}u), 
\end{equation}
where
\begin{equation}
u_{j}:=\partial u/\partial z^j.
\end{equation}
Since $u$ is primitive, we know that $\omega_{n-k+1}u\equiv 0$, thus $$\omega_{n-k+1}\theta_{z^j}u\equiv 0,$$
 which implies that the primitive decomposition of $\theta_{z^j}u$ contains at most three terms. Thus we can write
\begin{equation}
\theta_{z^j}u=a+\omega b+\omega^2 c,
\end{equation}
where $a, b, c$ are primitive. Thus we have
\begin{eqnarray*}
\tau *(\partial/\partial z^j) *(\omega_r u) & = & (-1)^k  *(\partial/\partial z^j) *(\omega_r u) \\
& = & i^{k^2}(-1)^{p+k} * \omega_{n-r-k} u_{j} -\omega_{r-1}a
+(n-r-k)\omega_r b \\
&  & - (n-r-k)(n-r-k+1) )\omega_{r+1} c.
\end{eqnarray*}
On the other hand, we have
\begin{equation}
\theta_{z^j}\omega_r u=\omega_r(a+\omega b+\omega^2 c). 
\end{equation}
Thus by \textbf{Fact 1}, we have
\begin{eqnarray*}
\Lambda\theta_{z^j}\omega_r u & = & (n-r-k-1)\omega_{r-1} a +(r+1)(n-r-k)\omega_{r} b  \\
&  & + (r+1)(r+2)(n-r-k)(n-r-k+1)\omega_{r+1} c.
\end{eqnarray*}
and
\begin{eqnarray*}
\theta_{z^j}\Lambda \omega_r u & = & (n-r-k+1) \theta_{z^j} \omega_{r-1} u  \\
& =  &  (n-r-k+1)\omega_{r-1} (a+\omega b+\omega^2 c).
\end{eqnarray*}
Thus
\begin{equation}
[\Lambda, \theta_{z^j}]  \omega_r u=-2 \omega_{r-1} a+(n-2r-k) \omega_r b+(2r+2)(n+1-r-k)\omega_{r+1} c.
\end{equation}
Now we have
\begin{eqnarray*}
A  & := &  \tau *(\partial/\partial z^j) *(\omega_r u)- [\Lambda, \theta_{z^j}]  \omega_r u \\
& =  &  i^{k^2}(-1)^{p+k} * \omega_{n-r-k} u_{j} + \omega_{r-1} a+ r \omega_r b  \\
&   & -(n+r-k+2)(n-r-k+1) \omega_{r+1} c.
\end{eqnarray*}
On the other hand, we have
\begin{eqnarray*}
B  & := &  (\partial/\partial z^j) (\omega_r u) \\
& =  &  \omega_{r-1}(a+\omega b+\omega^2 c)+\omega_r u_j \\
& =  &   \omega_{r-1}a+ r \omega_r b + r(r+1) \omega_{r+1} c + \omega_r u_j.
\end{eqnarray*}
Thus we have
\begin{eqnarray*}
A-B  & = &  i^{k^2}(-1)^{p+k} * \omega_{n-r-k} u_{j} - \omega_r u_j \\
&   &  - [(n+r-k+2)(n-r-k+1)+r(r+1)] \omega_{r+1} c .
\end{eqnarray*}
But since $u$ is primitive, we have
\begin{equation}
\omega_{n-k+1} u\equiv 0,
\end{equation}
which implies that
\begin{equation}
\theta_{z^j} \omega_{n-k} u + \omega_{n-k+1} u_j \equiv 0.
\end{equation}
Notice that 
\begin{equation}
\theta_{z^j} \omega_{n-k} u=\omega_{n-k} (a+\omega b+\omega^2 c) =\omega_{n-k}\omega^2 c.
\end{equation}
Thus we have
\begin{equation}
\omega_{n-k+1} (u_j +(n-k+1) \omega c) \equiv 0,
\end{equation}
which implies the primitivity of 
\begin{equation}
v:=u_j +(n-k+1) \omega c.
\end{equation}
Now we have
\begin{eqnarray*}
i^{k^2}(-1)^{p+k} * \omega_{n-r-k} u_{j}  & = & i^{k^2}(-1)^{p+k} * \omega_{n-r-k} (v-(n-k+1) \omega c) \\
&  = & \omega_r v + (n-k+1) (n-r-k+1) \omega_{r+1} c .
\end{eqnarray*}
and
\begin{equation}
\omega_{r} u_{j}=\omega_{r} (v-(n-k+1) \omega c)
= \omega_r v- (n-k+1)(r+1) \omega_{r+1} c.
\end{equation}
The above two identities imply that $i^{k^2}(-1)^{p+k} * \omega_{n-r-k} u_{j} - \omega_r u_j$ can be written as
\begin{equation}
[(n-k+1) (n-r-k+1)+ (n-k+1)(r+1)] \omega_{r+1} c,
\end{equation}
which implies $A=B$ together with our formula for $A-B$.
The proof is complete.
\end{proof}

Now let us prove Proposition \ref{pr:main-1}. We claim that it is enough to prove
\begin{equation}\label{eq:3.1}
\theta_{z^k}^*=-\frac12 [\Lambda, [\Lambda, \overline{\theta_{z^k}}]].
\end{equation}
In fact, $\theta_{z^k}=[\partial/\partial z^k, \omega]$ implies that $[\Lambda, (\partial/ \partial z^k)^*] =\theta_{z^k}^*$. Now our fundamental identity implies that
\begin{equation}
-(\partial/\partial z^k)^*=[\Lambda, \overline{\theta_{z^k}}]+\partial / \partial \bar z^k.
\end{equation}
Hence
\begin{equation}
\theta_{z^k}^*=-[\Lambda, [\Lambda, \overline{\theta_{z^k}}]]-[\Lambda, \partial / \partial \bar z^k].
\end{equation}
Thus  \eqref{eq:3.1} implies that
\begin{equation}
\theta_{z^k}^*=[\Lambda,\partial / \partial \bar z^k],
\end{equation}
and our claim is true.

\begin{proof}[Proof of Proposition \ref{pr:main-1}]  Now let us prove \eqref{eq:3.1}, which is easy since there are no derivatives involved. Since $\theta_{z^k}^*=\tau *\overline{\theta_{z^k}}*$, it suffices to show
\begin{equation}
\tau *\overline{\theta_{z^k}}*=-\frac12 [\Lambda, [\Lambda, \overline{\theta_{z^k}}]].
\end{equation}
If $u$ is a primitive $(p,q)$-class then by \textbf{Fact 2}, we have
\begin{equation}
\overline{\theta_{z^k}}*(\omega_r u):= i^{k^2} (-1)^{p} \omega_{n-r-k} \overline{\theta_{z^k}} u, \ k:=p+q.
\end{equation}
As before, we can write
\begin{equation}
\overline{\theta_{z^k}} u=a+\omega b+\omega^2 c,
\end{equation}
where $a, b, c$ are primitive. Put
\begin{equation}
M:=n-r-k.
\end{equation}
Then
\begin{equation}
\tau *\overline{\theta_{z^k}}*=-\omega_{r-2}a+(M+1)\omega_{r-1}b-(M+1)(M+2)c.
\end{equation}
On the other hand, notice that
\begin{equation}
-\frac12 [\Lambda, [\Lambda, \overline{\theta_{z^k}}]]=\Lambda \overline{\theta_{z^k}} \Lambda-\frac12 (\Lambda^2\overline{\theta_{z^k}}+\overline{\theta_{z^k}}\Lambda^2).
\end{equation}
Now 
\begin{equation}
\Lambda^2\overline{\theta_{z^k}} \omega_r u =\Lambda^2 \omega_r(a+\omega b+\omega^2 c).
\end{equation}
By \textbf{Fact 1}, we have
\begin{equation*}
\Lambda \omega_r(a+\omega b+\omega^2 c)=(M-1) \omega_{r-1}a+(r+1) M\omega_{r}b+(r+1)(r+2)(M+1) \omega_{r+1}c,
\end{equation*}
thus
\begin{eqnarray*}
\Lambda^2\overline{\theta_{z^k}} \omega_r u & = & (M-1)M \omega_{r-2}a+(r+1) M(M+1)\omega_{r-1}b \\
   &  & +(r+1)(r+2)(M+1)(M+2) \omega_{r}c.
\end{eqnarray*}
And we also have
\begin{equation}
\overline{\theta_{z^k}}\Lambda^2  \omega_r u=(M+1)(M+2)(\omega_{r-2}a+ (r-1)\omega_{r-1}b +r(r-1)\omega_{r}c),
\end{equation}
and
\begin{equation}
\Lambda \overline{\theta_{z^k}}\Lambda  \omega_r u=(M+1)(M\omega_{r-2}a+ r(M+1)\omega_{r-1}b +r(r+1)(M+2)\omega_{r}c).
\end{equation}
Thus 
\begin{equation}
-\frac12 [\Lambda, [\Lambda, \overline{\theta_{z^k}}]]=-\omega_{r-2}a +(M+1)\omega_{r-1}b-(M+1)(M+2)c=\tau *\overline{\theta_{z^k}}*.
\end{equation}
The proof is complete.
\end{proof}

\end{document}